\documentclass[10pt]{amsart}

\usepackage{color,graphicx,amssymb}
\usepackage{amsthm}
\usepackage{dsfont}
\usepackage{stmaryrd}
\usepackage{youngtab}

\numberwithin{equation}{section}

\theoremstyle{plain}
\newtheorem{theorem}{Theorem}[section]
\newtheorem{lemma}[theorem]{Lemma}
\newtheorem{proposition}[theorem]{Proposition}
\newtheorem{corollary}[theorem]{Corollary}

\newtheorem{conjecture}[theorem]{Conjecture}

\theoremstyle{definition}
\newtheorem{definition}[theorem]{Definition}

\newtheorem{question}[theorem]{Question}
\newtheorem{acknowledgements}[theorem]{Acknowledgements}

\begin{document}


\title{Complexity in Young's Lattice}
\author{Alexander Wires}
\address{School of Economic Mathematics, Southwestern University of Finance and Economics\\ 555 Liutai Avenue, Wenjiang District\\
Chengdu 611130, Sichuan, China}
\email{awires@swufe.edu.cn}
\date{July 20, 2019}



\begin{abstract}
We investigate the complexity of the partial order relation of Young's lattice. The definable relations are characterized by establishing the maximal definability property modulo the single automorphism given by conjugation; consequently, as an ordered set Young's lattice has an undecidable elementary theory and is inherently non-finitely axiomatizable but every ideal generates a finitely axiomatizable universal class of equivalence relations. We end with conjectures concerning the complexities of the $\Sigma_1$ and $\Sigma_2$-theories.
\end{abstract}

\thanks{The author was supported in part by National Natural Science Foundation of China Research Fund for International Young Scientists \#11650110429}

\maketitle


\section{Introduction}

Young's lattice is the lattice of non-negative integer partitions ordered by inclusion of their associated Young diagrams (or Ferrers diagrams) where the smallest element in the order is represented by the empty partition. A large measure of the research into Young's lattice reflects its classical and deep connection to representations of the finite symmetric groups. Focusing on the partial-order itself, Young's lattice  serves as a prominent example in the theory of differential posets contained in the seminal paper of R.P. Stanley \cite{stanley}. We find in the paper of Suter \cite{suter} an illustration of how intersections of certain principal ideals display dihedral automorphisms. We are led to wonder what other complexities may be hidden in the order relation of Young's lattice? As one approach to the question, in this paper we consider the complexity of Young's lattice from a logical perspective. Partial motivation may also be found in similar investigations into the complexity of the various finite alphabet word orders in Kuske \cite{kuske} and Halfon, Schnoebelen and Zetzsche \cite{halfon}.

We consider Young's lattice as an ordered set and seek to characterize the first-order definable relations in this order. Transposition of the Young diagrams of partitions induces an automorphism of the order relation which is traditionally referred to as conjugation. After adding a single constant to the language of the ordered set to account for this automorphism, we show that Young's lattice has a certain bi-interpretation with natural arithmetic (Theorem~\ref{thm:mdp}) called the maximal definability property Kudinov, Selivanov and Yartseva \cite{kudinov}; consequently, the elementary theory is undecidable and inherently non-finitely axiomatizable. One step in this direction is accomplished by showing there is an interpretation of arithmetic utilizing $\Pi_3$-formulas and a single constant (Theorem~\ref{thm:sigma1}). 

We would like to have a characterization of certain small fragments of the first-order theory of Young's lattice; for example, we would like to characterize the complexity of the $\Sigma_1$-theory with or without constants and the $\Sigma_2$-theory with a single constant for the partition $1+1$. While it would be desirable to utilize the analogous results already attained for small fragments of the subword order with constants over two-letter alphabets \cite[Thm 3.3, Cor 3.6]{halfon}, the exact link between the two orders remains undetermined and we end by stating our conjectures.


Even though the results in this paper are about first-order definable relations, our arguments may be said to be combinatorial in that we do not require any specialized knowledge of mathematical logic nor model theory and may be read conveniently without any such background. Concentrating more on the combinatorial aspect, the topic and arguments of this paper can be pursued for the wider class of differential posets where it may be possible to prove similar results. 

In a broader setting, the topic of definability in the order relation of Young's lattice is related to the subject of positive definability in the substructure relation of finite structures in Je\v{z}ek and McKenzie \cite{semilat,poset,distr,latt} and Wires \cite{wires}. For a fixed class of structures, the isomorphic substructure relation defines a partial-order on the isomorphism types of the finite structures in that class. We can then view any analysis or characterization of the first-order definable relations in this ordered set as yielding a characterization of the first-order expressive power of the substructure relation applied to the finite structures. Since an unordered partition of a positive integer can also be interpreted as a finite equivalence relation, the order relation in Young's lattice is the same as the isomorphic substructure relation among finite equivalence relations. This means that the results of this paper can be seen as establishing positive definability for the class of equivalence relations. Since it is easy to see that the partial order on partitions is a well quasi-order, it follows from A.I. Mal'cev \cite{malcev} that the universal class generated by any ideal is finitely axiomatizable; that is, finitely axiomatizable by universal sentences. This complements similar work on ordered structures like posets and distributive lattices \cite{semilat,poset,distr,latt} and the unordered structure of simple graphs \cite{wires} and Thinniyam \cite{thinn3}.

After setting notation and reviewing preliminaries in Section~\ref{sec:2}, the interpretation of arithmetic is developed in Section~\ref{sec:3} and the maximal definability property is then established in Section~\ref{sec:4}. We end in Section~\ref{sec:5} with a few open questions.


\section{Preliminaries}\label{sec:2}

Arithmetic partitions will be denoted by lower-case Greek letters. If $\pi$ is a partition of $n \in \mathds{N}$, then we can represent $\pi$ as a finite sequence $\pi=(n_1,\ldots,n_t,)$ where $n=n_1 + \cdots + n_t$, each summand $n_{i}$ is called a \emph{part} and $1 \leq n_{i+1} \leq n_{i}$. It is standard to define the empty partition as the unique partition which has no parts. Each partition has a corresponding Young diagram and the partial-order $\leq$ determined by containment of the diagrams can be equivalently described using the descending sequence representation: given $\pi=(n_1,\ldots,n_t)$ and $\sigma=(s_1,\ldots,s_r)$, we have
\begin{align}\label{eqn:1}
\sigma \leq \pi \quad \Leftrightarrow \quad r \leq t \ \text{ and } \ s_i \leq n_i \ \text{ for all } \ i \leq r.
\end{align}
From the above, it is easy to see that $\leq$ is indeed a partial-order on the set of partitions $\mathcal P$, and the ordered set $\textbf{Y}=\left\langle \mathcal P,\leq \right\rangle$ is known as \emph{Young's lattice}. As the name suggests, $\textbf{Y}$ is in fact a lattice, a distributive lattice to be precise, but this fact will not be instrumental in our development.

If $\sigma < \pi$ but there does not exist $\rho$ such that $\sigma < \rho < \pi$, then we write $\sigma \prec \pi$ and say $\pi$ covers $\sigma$, or that $\sigma$ is a subcover of $\pi$. We write $|\pi|=n$ if $\pi$ is a partition of the positive integer $n$ and refer to $|\pi|$ as the cardinality of $\pi$. It is immediate from \eqref{eqn:1} that $\pi \prec \rho$ if and only if $\left|\pi\right| + 1= \left|\rho\right|$ and $\pi \leq \rho$. It follows that $\textbf{Y}$ is graded by cardinality in that partitions have the same cardinality if and only if they are at the same height in $\textbf{Y}$.

For any partition $\pi \in \mathcal P$, the transpose of the corresponding Young diagram defined by interchanging the rows and columns produces the Young diagram of another partition denoted by $\pi^{\partial}$ and called the \emph{conjugate} of $\pi$. The conjugation map $\pi \mapsto \pi^{\partial}$ is an automorphism of $\textbf{Y}$.

For our needs, it will be necessary to deviate from standard notation and introduce a different summation representation for integer partitions. It will be important to have a convenient way of recording the number of times a part of a given size appears in the partition. The \emph{canonical representation} of $\pi \in \mathcal P$ is $\pi = \sum_{i=1}^{t} m_i[n_i]$ where $[n_i]$ stands for a part $n_i$ in $\pi$ and the coefficient $m_i$ means the part $n_i$ appears exactly $m_i$ times in the partition; for example, the partition $\pi=(6,6,5,4,4,4,3,3,2,1,1,1,1)$ has the canonical representation
\[
\pi = 2[6] + [5] + 3[4] + 2[3] + [2] + 4[1].
\]
Notice we still maintain the convention that $n_{i+1} < n_{i}$. It is clear that $n \geq m$ if and only if $[n] \geq [m]$, but there is a distinction in that the first inequality is among positive integers and the second is the partial order of Young's lattice. The notation is inspired by the dual role of a partition as an equivalence relation, and so the summand $m_i[n_i]$ reflects an equivalence block of cardinality $n_i$ which appears exactly $m_i$ times.

A great deal of our argument involves showing certain relations in $\textbf{Y}$ are first-order definable by formulas with a special syntax. We recommend Burris and Sankappanavar \cite[Sec V.1]{univ} for a quick review of the basic notions of structure and first-order satisfaction, but \cite{flum} as an accessible and thorough textbook. For a structure $\left\langle A, \tau \right\rangle$, a relation $R \subseteq A^{k}$ is first-order definable if there is a first-order formula $\phi(x_1,\ldots,x_k)$ with free variables among $\{x_1,\ldots,x_k\}$ such that
\[
(\pi_1,\ldots,\pi_k) \in R \quad \Leftrightarrow \quad \left\langle A, \tau \right\rangle \vDash \phi(\pi_1,\ldots,\pi_k).
\]
As an example, for any ordered set $\left\langle P, \leq \right\rangle$ the covering relation $\mathrm{Cov}=\{(x,y): x \prec y\}$ is first-order definable by the formula
\begin{align}\label{eqn:2}
\phi(x,y)_{\mathrm{Cov}}:=x \leq y \mathrel{\bigwedge} x \neq y \mathrel{\bigwedge} \forall z \left(x \leq z \leq y \longrightarrow x = z \bigvee y = z\right).
\end{align}
If $\left\langle P, \leq \right\rangle = \left\langle \mathds{N}, \leq \right\rangle$ then $\mathrm{Cov}=\{(n,n+1): n \in \mathds{N}\}$, but for the rationals $\mathrm{Cov}=\emptyset$. 

For any structure $\left\langle A, \tau \right\rangle$, let $\mathrm{Def}(A, \tau)$ denote the set of first-order definable relations in the structure.

A classic result states that any first-order formula is logically equivalent to a formula with an explicit syntax of the form
\begin{align}\label{eqn:3}
\bar{Q}_{n}\bar{y}_{n}\bar{Q}_{n-1}\bar{y}_{n-1} \cdots \bar{Q}_{1}\bar{y}_{1}\phi(x_1,\ldots,x_k,\bar{y}_1,\ldots,\bar{y}_n)
\end{align}
where
\begin{itemize}

	\item each $\bar{Q}_i$ is a finite sequence of a single quantifier $\exists$ or $\forall$,
	
	\item the quantifiers alternate in the index, and
	
	\item $\phi(x_1,\ldots,x_k,\bar{y}_1,\ldots,\bar{y}_n)$ is an \emph{open formula} - a formula which contains no quantifiers.

\end{itemize}
A formula is in \emph{prenex} form if the syntax has the form in Eqn~\ref{eqn:3}.	Define $\Sigma_n$	to be the set of formulas which have a prenex form where the quantifiers $\bar{Q}_n$ are all existential, and $\Pi_n$ are the formulas which have a prenex form where the quantifiers $\bar{Q}_n$ are all universal. Define $\Delta_n = \Sigma_n \cap \Pi_n$ and note $\Delta_0 = \Sigma_0= \Pi_0$ refers to the set of open formulas. The following inclusions
\[
\Pi_i \subseteq \Pi_{i+1} \quad , \quad  \Sigma_{i} \subseteq \Sigma_{i+1} \quad , \quad \Pi_{i-1} \subseteq \Sigma_{i} \subseteq \Pi_{i+1}
\] 
are immediate from the definitions and it can naively be said that complexity increases with the indices. Definability can then be restricted to the syntax; for example, a relation is $\Sigma_n$-definable if it is definable by a $\Sigma_n$-formula. The formula in Eqn~\ref{eqn:2} shows the covering relation is $\Pi_1$-definable in any poset.

One way in which more complex formulas are produced	is through the use of implication; for example, suppose a subset of partitions $A \subseteq \mathcal P$ is definable by a $\Pi_1$-formula	$\psi(x)$ and $\pi \in \mathcal P$ is a fixed partition. Then
\begin{align}\label{eqn:4}
 \psi(x) \mathrel{\bigwedge} x \leq \pi \mathrel{\bigwedge} \left( \forall y \left( \psi(y) \mathrel{\bigwedge} y \leq \pi  \longrightarrow  y \leq x \right)  \right) 
\end{align}
defines the set of partitions which are maximal among the partitions in $A$ which are below $\pi$. This is now a $\Pi_2$-formula since the implication is in the range of the universal quantifier for $y$, and because the universal quantifiers which are nested in $\psi(y)$ reside in the precedent of the implication, the implication is logically equivalent to a formula with existential quantifiers.

Fix a structure $\left\langle A, \tau \right\rangle$. For any set $B$ and map $\psi: B \rightarrow A$, we can define a $\tau$-structure on the set $B$ by pulling back the relations of $\left\langle A, \tau \right\rangle$ in the following manner: for each k-ary relation $R \in \tau$ define a new relation $R_{\phi}$ on $B$ by 
\[
(b_1,\ldots,b_k) \in R_{\phi} \quad \Leftrightarrow \quad  (\psi(b_1),\ldots,\psi(b_k)) \in R.
\]
We will be interested in the case when $\left\langle A, \tau \right\rangle = \left\langle \mathds{N}, +, \times \right\rangle$.

\begin{definition}(\cite{kudinov})
A structure $\left\langle A, \tau \right\rangle$ is \emph{arithmetic} if there is a bijection $\#: A \rightarrow \mathds{N}$ such that $\mathrm{Def}(A,\tau) \subseteq \mathrm{Def}(A, +_{\#},\times_{\#})$.

An arithmetic structure $\left\langle A, \tau \right\rangle$ has the \emph{maximal definability property} if $\mathrm{Def}(A, \tau) = \mathrm{Def}(A, +_{\#}, \times_{\#})$
\end{definition}

                            
\section{Arithmetic}\label{sec:3}

The main goal of this section is to establish a particular interpretation of natural arithmetic using formulas of small complexity culminating in Theorem~\ref{thm:sigma1}. A partition $\pi=[n]$ with a single part is called \emph{total}, while partitions of the form $\pi=m[1]$ in which each part has size one are \emph{trivial}. Since Young's lattice includes the empty partition, we declare that the empty partition is both total and trivial. In $\textbf{Y}$, the empty partition $\emptyset$ has the $\Pi_1$-definition $\forall y(x \leq y)$ since it is the smallest element. The trivial partition $[1]$ is the only cover of $\emptyset$ which is then $\Pi_1$-definable using Eqn~\ref{eqn:2}. Now the set $\{[2],[1]+[1]\}$ is $\Pi_1$-definable as the covers of $[1]$, but there can be no way to define either $[2]$ or $[1]+[1]$ separately in a first-order way since $[2]^{\partial}=[1]+[1]$. We shall add the constant $[1]+[1]$ to the language and refer to the new structure $\textbf{Y}^{\ast} = \left\langle  \mathcal P, \leq , [1]+[1] \right\rangle$. In what follows, definability will usually refer to formulas built from the partial-order $\leq$ of Young's lattice together with the particular partition $[1]+[1]$.

\begin{lemma}\label{ecomptriv}
The set of total partitions are $\Delta_0$-definable and the trivial partitions are $\Pi_1$-definable in $\textbf{Y}^{\ast}$
\end{lemma}   
\begin{proof}
A partition $\pi$ is total if and only if $\pi \ngeq [1]+[1]$. Since $\{[2],[1]+[1]\}$ is $\Pi_1$-definable, with the constant $[1]+[1]$ in the language, we have that $\{[2]\}$ is $\Pi_1$-definable. Then the set of trivial partitions consists of those $\pi$ such that $\pi \ngeq [2]$. 
\end{proof}

A partition $\pi$ \textit{rectangular} if all parts have the same size; thus, a rectangular partition has the form $\pi = m[n]$.

\begin{lemma}\label{lem:rect}
$\pi$ has a unique lower cover if and only if $\pi$ is rectangular; consequently, the set of uniform partitions is $\Delta_2$-definable without constants.
\end{lemma}
\begin{proof}
For necessity, notice that if $\pi = m[n]$, then $\rho \prec \pi$ implies $\rho = (m-1)[n] + [n-1]$.

If $\pi$ is not uniform, then we can write $\pi = [n] + \sum_{i=1}^{r} [s_i]$ where $n \geq s_i$ and there exists $s_k$ such that $n > s_k$. Then $\sigma=[n-1] + \sum_{i=1}^{r} [s_{i}]$ and $\rho=[n] + [s_{k}-1] + \sum_{i \neq k} [s_{i}]$ are subcovers of $\pi$ which are incomparable since $\rho$ contains the same number of parts of size $n$ that $\pi$ has, but $\sigma$ does not.

The set of rectangular partitions is then definable since the the formula
\[
\forall y \forall z \left( \phi_{\mathrm{Cov}}(y,x) \bigwedge \phi_{\mathrm{Cov}}(z,x) \longrightarrow y=z \right)
\]
defines the property that $x$ has a unique lower cover. This is a $\Pi_2$-definition since the covering relation is $\Pi_1$-definable and is contained in the precedent of the implication. That $x$ has a unique lower cover can also be defined by the $\Sigma_2$ definition
\[
 \exists x^{\ast} \left(x^{\ast} \leq x \bigwedge x \neq x \bigwedge \forall y \left( y \leq x \longrightarrow y \leq x^{\ast} \bigvee y = x \right) \right);
\]
therefore, the set of rectangular partitions is $\Delta_2$-definable.
\end{proof}

The next result has a simple proof if we first introduce the terminology of reconstruction. For any digraph $G$, a vertex-deleted subgraph is the induced subgraph which results after deleting a single vertex from the vertex set. Let $\mathcal H$ be a class of finite digraphs closed under taking induced subgraphs. The Strong Reconstruction Conjecture for $\mathcal H$ states that every digraph from $\mathcal H$ on at least four vertices is uniquely determined by its set of vertex-deleted subgraphs. The conjecture is known to fail when $\mathcal H$ is the full class of digraphs or tournaments \cite{stock}, but it was shown to hold by Pretzel and Siemons \cite{eqrec} for the class of equivalence relations. When viewed as integer partitions, Strong Reconstruction states that every partition of cardinality at least four is uniquely determined by its set of lower covers in Young's lattice.

\begin{proposition}\label{indivdefine}
Every partition is first-order definable in $\textbf{Y}^{\ast}$; consequently, the conjugation map is the unique nontrivial automorphism of Young's lattice.
\end{proposition}
\begin{proof}
First, we observe that partitions with cardinality at most three are $\Pi_1$-definable. By the first paragraph of this section, every partition of cardinality at most two has a $\Pi_1$-definition. The partitions $[3]$ and $[1]+[1]+[1]$ can then be recovered as having the unique subcovers $[2]$ and $[1]+[1]$, respectively. Then $[2]+[1]$ has both $[2]$ and $[1]+[1]$ as subcovers. 

Now assume, every partition at height $n-1\geq 3$ as a first-order definition in $\textbf{Y}^{\ast}$. Suppose $\pi$ has height $n$ and let $\{\sigma_1,\ldots,\sigma_k\}$ be the set of lower covers of $\pi$. By induction, there are formulas $\phi_{1}(x_{1}),\ldots,\phi_{k}(x_{k})$ such that $\sigma_i$ is the unique element in Young's lattice which satisfies the formula $\phi_{i}(x_{i})$. Then Strong Reconstruction implies that $\pi$ uniquely satisfies $\forall y \left(\mathrm{Cov}(y,x) \longleftrightarrow \phi_1(y) \vee \cdots \vee \phi_k(y) \right)$.

We show conjugation $\partial:\textbf{Y} \rightarrow \textbf{Y}$ is the unique non-trivial automorphism of $\textbf{Y}$. Suppose $f$ is another automorphism. Since the set $\{[2],[1]+[1]\}$ is first-order definable, it is closed under $\tau$. Suppose $f([1]+[1])=[2]$. Then $f^{-1}\partial$ is an automorphism which fixes $[1]+[1]$. For any partition $\pi$, there is a first-order formula $\phi_{\pi}(x,y)$ such that $\pi$ is the unique element in $\textbf{Y}$ such that $\textbf{Y} \vDash \phi_{\pi}(\pi,[1]+[1])$. Let $R$ be the binary relation defined by the formula $\phi_{\pi}(x,y)$. Then $(\pi,[1]+[1]) \in R$ implies $(f^{-1}\partial(\pi),[1]+[1])= (f^{-1}\partial(\pi),f^{-1}\partial([1]+[1])) \in R$; thus, by uniqueness we must have $f^{-1}\partial(\pi)=\pi$, and so $\partial(\pi)=f(\pi)$. This implies $f=\partial$. If it were the case that $f$ fixes $[1]+[1]$, then the same argument would show $f$ is the identity map.
\end{proof}

We define two functions. For a partition $\pi$, let $l(\pi)$ equal the number of parts in $\pi$. This will be referred to as the length of the partition. Set $b(\pi)=n$ if $[n]$ is the largest part in a partition.

\begin{lemma}\label{lem:unif}
We have the following:
\begin{enumerate}

	\item $\left\{ (\rho,\pi) : \rho = m[1], l(\pi) = m, m \geq 1 \right\}$ is $\Pi_1$-definable in $\textbf{Y}^{\ast}$;
	
	\item $\left\{ (\rho,\pi) : \rho = [n] \hbox{ all parts of } \pi \hbox{ have at most $n$ elements }\right\}$ is $\Delta_0$-definable in $\textbf{Y}^{\ast}$;
	
	\item $\left\{ (\rho,\sigma,\pi) : \rho = [m], \sigma = n[1], \pi = n[m] \right\}$ is $\Delta_2$-definable in $\textbf{Y}^{\ast}$.

\end{enumerate}
\end{lemma}
\begin{proof}
(1) We see that $l(\pi) = m$ iff $m[1] \leq \pi$ but $(m+1)[1] \nleq \pi$.

(2) That every block of $\pi$ has at most $n$ elements is given by the condition $[m] \nleq \pi$ for $[m] > [n]$.

(3) We see that $(\rho,\sigma,\pi)$ is in this relation if and only if $\rho$ is total with $b(\pi)=|\rho|$, $\sigma$ is trivial with $l(\pi)=|\sigma|$, and $\pi$ is rectangular.

It is immediate that $\pi \approx n[m]$ satisfies the condition. To see that they are sufficient, we must have $\pi \approx r[t]$ by rectangularity, $l(\pi)=|\sigma|$ implies $r=|\sigma|$, and $t=|\rho|$ since $\rho$ is the largest total partition below $\pi$.
\end{proof}

\begin{proposition}\label{edistinct}
$\left\{ \pi : \hbox{ all parts of } \pi \hbox{ are distinct } \right\}$ is $\Pi_2$-definable in $\textbf{Y}^{\ast}$.
\end{proposition}
\begin{proof}
Let $b(\pi)=n$ and $l(\pi)=t$. Then all the blocks of $\pi$ are distinct if and only if $\forall s < t$ and for all $[n_{s}] \leq [n]$ such that $s[n_{s}] \leq \pi$ and $s[n_{s}+1] \nleq \pi$, then we must have $(s+1)[n_{s}] \nleq \pi$.

First, suppose all parts of $\pi$ are distinct and order them as $[n_{1}] > [n_{2}] > \cdots > [n_{t}]$ where $n_1 = n$. For $s \leq t$ and a rectangular partition $s[p] \leq \pi$ such that $s[p+1] \nleq \pi$, then we must have $p=n_s$. Since $n_s > n_{s+1}$, it is the case that $(s+1)[p] \nleq \pi$.

Conversely, suppose $\pi$ satisfies the conditions and consider the non-canonical representation of $\pi$ with $[n_{1}] \geq [n_{2}] \geq \cdots \geq [n_{t}]$ where $n_1 = n$. For a contradiction, suppose there is an interval in the index with repeated parts $n_{k}=n_{k+1}= \cdots = n_{k+j}$ with $j \geq 1$. Then we have the rectangular partition $k[n_{k}] \leq \pi$ with $k[n_{k}+1] \nleq \pi$, but $(k+1)[n_{k}] \leq \pi$ because $[n_{k}]=[n_{k+1}]$ - a contradiction of the conditions. It must be that all the parts of $\pi$ are distinct.

The defining condition above is in the form of an implication which is in the range of a universal quantifier. Since rectangular partitions are $\Delta_2$-definable, we can use a $\Sigma_2$-formula for them in the precedent of the implication so that the whole formula is logically equivalent to a $\Pi_2$-formula.
\end{proof}

We can now specify the existence of a particular part.

\begin{proposition}\label{eblockus}
$\left\{ (\rho,\pi): \rho = [n] \hbox{ and } [n] \hbox{ is a part of  } \pi \right\}$ is $\Pi_2$-definable in $\textbf{Y}^{\ast}$.
\end{proposition}
\begin{proof}
$(\rho,\pi)$ is in this relation if and only if $\rho = [n]$, $\rho \leq \pi$, and whenever $r[n] \leq \pi$ but $(r+1)[n] \nleq \pi$, then $r[n+1] \leq \pi$.

We only argue sufficiency. Suppose $\pi$ satisfies the above condition and write the canonical representation $\pi \approx \sum_{i=1}^{t}  m_{i}[n_{i}]$. Since $\rho =[n] \leq \pi$, take $k$ largest such that $[n] \nleq [n_{k+1}]$. The largest rectangular partition with parts of size $n_k$ below $\pi$ is $\left(\sum_{i=1}^{k} m_{i}\right)[n_{k}]$. Set $r=\sum_{i=1}^{k} m_{i}$. Then $r[n] \leq r[n_k] \leq \pi$, but $(r+1)[n] \leq \pi$ because $[n_{k+1}] < [n]$. According to the $\Pi_1$-definition, we must have $r[n+1] \nleq \pi$ which can only happen if $n=n_k$. 
\end{proof}

\begin{definition}
For $n \geq 1$, a partition $\sigma \approx \sum_{i=1}^{n} [i]$ is called a \textsl{factorial} and will be denoted as $[n]!$ (In the literature, such partitions a often called triangular, but we will use a different nomenclature because of the role they play in defining multiplication).
\end{definition}

Our approach to the definability of arithmetic is to first show that factorials are definable.

\begin{proposition}\label{prop:factorial}
$\left\{(\rho,\pi): \rho \approx [n], \pi = [n]! \right\}$ is $\Pi_2$-definable in $\textbf{Y}^{\ast}$.
\end{proposition}
\begin{proof}
The claim is that $\pi=[n]!$ if and only if
\begin{enumerate}

	\item $\rho$ is total and $b(\pi)=b(\rho)=n$;
	
	\item for all $[r] \leq [n]$ we have that $[r]$ is a part of $\pi$;
	
	\item all the parts of $\pi$ are distinct.

\end{enumerate}

If $\pi \approx [n]!$, then it is easy to see the conditions are satisfied.

Suppose $\pi$ satisfies conditions (1) - (3). Conditions (1) and (2) imply $\pi = \sum_{i=1}^{n} m_{i}[i]$, and condition (3) implies each $m_i = 1$.
\end{proof}

We can now define the pairs of total and trivial partitions which are at the same height.

\begin{lemma}\label{epropht}
$\left\{(\rho,n): \rho \hbox{ is total}, \hbox{ $\pi$ is trivial}, \left|\rho\right| = \left|\pi\right| \right\}$ is $\Pi_3$-definable in $\textbf{Y}^{\ast}$.
\end{lemma}
\begin{proof}
$(\rho,\pi)$ is in this relation if and only if $\rho = [r]$, $\pi = m[1]$, and $l(\sigma) = m$ where $\sigma = [r]!$. The last requirement is an implication which asserts a factorial in the precedent. The definition is then logically equivalent to a $\Pi_3$-formula using Proposition~\ref{prop:factorial}.
\end{proof}

With factorials, we don't have to start counting the parts just from $[1]$ - we can now perform addition.

\begin{proposition}\label{eaddition}
$\left\{(\rho,\sigma,\pi): \rho,\sigma,\pi \hbox{ are total and } \left|\rho\right| + \left|\sigma\right|  = \left|\pi\right| \right\}$ is $\Pi_3$-definable in $\textbf{Y}^{\ast}$.
\end{proposition}
\begin{proof}
$(\rho,\sigma,\pi)$ is in this relation if and only if 
\begin{enumerate}

	\item $\rho,\sigma,\pi$ are total, 
	
	\item $\rho,\sigma < \pi$, and
	
	\item for all $\beta^{\ast}$, for all $\alpha$, if all parts of $\beta^{\ast}$ are distinct, $\alpha$ is total, and $\alpha$ is a part of $\beta{\ast}$ if and only if $\rho < \alpha \leq \pi$, then it must be that $l(\beta^{\ast}) \geq|\sigma|$. 
	
\end{enumerate} 

The condition that all parts of $\beta^{\ast}$ are distinct is $\Pi_2$ from Proposition~\ref{edistinct}. The bi-implication that $\alpha$ is a part of $\beta{\ast}$ if and only if $\rho < \alpha \leq \pi$ is logically equivalent to a $\Sigma_2$-formula using Proposition~\ref{eblockus}. Since both conditions are in the precedent of the implication in condition (3), the whole condition is logically equivalent to a $\Pi_3$-formula. 
\end{proof}

It follows from Lemma~\ref{epropht} and Proposition~\ref{eaddition} that we can also interpret addition by considering the corresponding triplets of trivial partitions.

We may refer to a partition of the form $m[n]$ as $n$-rectangular to denote the fact that all the parts have size $n$. We will also say $m[n]$ has \textit{frequency} $m$ to refer to the part $[n]$ appearing $m$ times. We saw in Lemma~\ref{lem:unif} that the set of $n$-rectangular partitions is definable; moreover, it is easy to see that the they are linearly ordered. The next result allows us to pick out the rectangular partitions which appear in a canonical representation.

\begin{proposition}\label{prop:freq}
The relation 
\[
\left\{(\rho,\sigma,\pi): \sigma = [n]  \hbox{ and $\rho$ is a part of } \pi \hbox{ which appears exactly $n$ times } \right\}
\]
is $\Pi_3$-definable in $\textbf{Y}^{\ast}$.
\end{proposition}
\begin{proof}
Let $\rho = [r]$ and $\sigma = [n]$. We have that $[r]$ is a block of $\pi$ which appears exactly $n$ times if and only if $\pi \approx n[r]$, or
\begin{enumerate}

	\item $[r]$ is a part of $\pi$;
	
	\item If $b(\pi) = r$, then $n[r]$ is the maximal $r$-rectangular partition below $\pi$;
	
	\item If $b(\pi) \neq r$ and for all $\beta^{\ast}$ such that
	
		\begin{enumerate}
		
			\item $\beta^{\ast}$ is a part of $\pi$ with $\beta^{\ast} > [r]$, and
			
			\item where $m[r]$ is the maximal $r$-rectangular partition below $\pi$, and
			
			\item $t\beta^{\ast}$ is the maximal $|\beta^{\ast}|$-rectangular partition below $\pi$
		
		\end{enumerate}
		then $n \leq m-t$.

\end{enumerate}
If we examine the canonical representation of $\pi = \sum_{i=1}^{s} m_{i}[n_{i}]$, then for any part $[n_{r}]$, we see that $\left(\sum_{i=1}^{r} m_{i}\right)[n_{r}]$ is the largest $n_{r}$-rectangular partition below $\pi$, and so the correctness of the above characterization follows since $m_r = \sum_{i=1}^{r} m_{i} - \sum_{i=1}^{r-1} m_{i}$. 

Condition (1) is a $\Pi_2$-definition by Proposition~\ref{eblockus}. In condition (2), we are asserting a maximal r-rectangular partition below $\pi$. If we use a $\Sigma_2$-formula for the rectangular property, then the discussion following Eqn.~\ref{eqn:4} concludes in this case that it is logically equivalent to a $\Pi_2$-formula. For the same reason, (3b) and (3c) are both $\Pi_2$-formulas and (3a) is also a $\Pi_2$-formula by Proposition~\ref{eblockus}; altogether, the implication (3) is logically equivalent to a $\Pi_3$-formula.  
\end{proof}

\begin{proposition}\label{prop:compare}
$\left\{ (\rho, \pi): \rho \hbox{ is total and } \left|\pi\right| \geq \left|\rho\right| \right\}$ is $\Pi_3$-definable in $\textbf{Y}^{\ast}$.
\end{proposition}
\begin{proof}
The claim is that $(\rho,\pi)$ is in the relation if and only if $\rho$ is total, and for any partition $\sigma$ which satisfies the conditions below, we have $l(\sigma) \geq |\rho|$:
\begin{itemize}

	\item[($\ast\ast$)] If $[r] \leq \pi$ and $m[r] \leq \pi$ but $(m+1)[r] \nleq \pi$ for some $m$, then $[r]$ is a part of $\sigma$ which appears at least $m$ times.

\end{itemize}

To verify necessity, let $\pi \approx \sum_{i=1}^{t} m_{i}[n_{i}]$ with $|\pi| \geq n = |\rho|$, and suppose $\sigma$ is a partition which satisfies the condition ($\ast\ast$). We wish to show $l(\sigma) \geq n$. Set $M_r = \sum_{i=1}^{r} m_i$ and note that $M_1 < M_2 < \cdots < M_t$. For each $n_{i+1} < r \leq n_{i}$, we see that $M_{i}[r]$ is a maximal $r$-rectangular partition below $\pi$ and so $(\ast\ast)$ implies $[r]$ is a part of $\sigma$ which appears at least $M_{i}$ times; altogether, it must be the case that 
\begin{align*}
\sigma \geq  \sum_{r \leq n_{t}} M_{t}[r]  + \sum_{i=1}^{t-1} \sum_{n_{i+1} < r \leq n_{i}} M_{i}[r] = \sum_{i=1}^{t} m_i[n_i]!
\end{align*}
This implies then that 
\[
l(\sigma) \geq l\left(\sum_{i=1}^{t} m_i[n_i]!\right) = \sum_{i=1}^{t} m_{i}n_{i} = |\pi| \geq n.
\]

To establish sufficiency, suppose $\rho$ is total but $|\pi| < n = |\rho|$. Let $\pi = \sum_{i=1}^{t} m_{i}[n_{i}]$. Set $\sigma = \sum_{i=1}^{t} m_{i}[n_{i}]!$ and observe that $l(\sigma) = \sum_{i=1}^{t} m_{i}n_{i} < n$. Suppose $[s] \leq \pi$ and let $k$ be the smallest index for which $[s] \leq [n_{k}]$. For the part $[n_{k}]$, we see that $\left(\sum_{i=1}^{k} m_{i} \right)[n_{k}] \leq \pi$ is maximal. If $r[s] \nleq \pi$ for $r > \sum_{i=1}^{k} m_{i}$, then by definition of the canonical representation, we must have $[s] \nleq [n_{k+1}]$ which contradicts the choice of $[n_{k}]$; therefore, $\left(\sum_{i=1}^{k} m_{i}\right)[s] \leq \pi$ is maximal among $s$-rectangular partitions. Notice that $[n_{k}]$ appears in $\sigma$ for each factorial $[n_r]!$ where $n_r > n_k$; that is, $[n_{k}]$ appears $\sum_{i=1}^{k} m_{i}$ times which is exactly how often $[s]$ appears as a part in $\sigma$. We have shown the partition $\sigma$ satsifies ($\ast\ast$).
\end{proof}

We note in passing that the previous argument essentially shows the $\Pi_3$-definability of $\sigma =\sum_{i=1}^{t} m_i[n_i]!$ given $\pi = \sum_{i=1}^{t} m_{i}[n_{i}]$.

\begin{proposition}
$\left\{ (\rho, \pi): \rho \hbox{ is complete and } \left|\pi\right| = \left|\rho\right| \right\}$ is $\Pi_3$-definable in $\textbf{Y}^{\ast}$.
\end{proposition}
\begin{proof}
Using the $\Pi_3$-definition in Proposition~\ref{prop:compare}, we would have $\left|\pi\right| \geq \left|k\right|$ but $\left|\pi\right| \ngeq \left|k\right| + 1$.
\end{proof}

We can now interpret multiplication.

\begin{proposition}\label{emultipli}
$\left\{(\rho,\sigma,\pi): k,\rho,\pi \hbox{ are total and } \left|\pi\right|  = \left|\rho\right| \left|\sigma\right| \right\}$ is $\Pi_3$-definable in $\textbf{Y}^{\ast}$.
\end{proposition}
\begin{proof}
$(\rho,\sigma,\pi)$ is in this relation if and only if $\rho,\sigma,\pi$ are total and $|\pi| = |\beta|$ where $\beta \approx |\rho|[|\sigma|]$.
\end{proof}

Let $\left\langle \mathds{N}, +, \times \right\rangle$ denote the structure over the set of non-negative integers such that the operations of addition and multiplication have their usual meaning. Propositions~\ref{ecomptriv}, \ref{eaddition}, and \ref{emultipli} state that we have an interpretation of $\left\langle \mathds{N}, +, \times \right\rangle$ into $\textbf{Y}^{\ast}$ in which the ternary relations for addition and multiplication are definable over the total partitions by $\Pi_3$-formulas. Undecidability of the positive $\Sigma_1$-theory of $\left\langle \mathds{N}, +, \times \right\rangle$ established in Matiyasevich \cite{matiy} yields the following:

\begin{theorem}\label{thm:sigma1}
The $\Sigma_4$-theory of $\textbf{Y}^{\ast}$ is undecidable.	
\end{theorem}

Since the elementary theory of a fixed structure is complete, by \cite[Thm 1, Thm 7, Thm 10]{interp} the above interpretation establishes the following:

\begin{theorem}\label{thm:elementary}
The elementary theory of Young's lattice is undecidable and inherently non-finitely axiomatizable.	
\end{theorem}

                            
\section{Maximal Definability Property}\label{sec:4}

In this section, we establish the maximal definability property for $\textbf{Y}^{\ast}$. Enumerate the primes $\{p_1,p_2,p_3,\ldots\}=\{2,3,5,\ldots\}$.

\begin{definition}
Define $\#: \mathcal P \rightarrow \mathds{N}$ in the following manner:
\begin{enumerate}

	\item $\#\emptyset=0$
	
	\item If $\sigma= \sum_{i=1}^{k} m_i[n_i]$, then 
\[
 \#\pi = 
  \begin{cases} 
   \prod_{i=1}^k p^{m_i}_{n_i} & \text{if some } n_i \neq 1  \\
   p^{m-1}_1       & \text{if } \sigma=m[1]
  \end{cases}
\]

\end{enumerate}
\end{definition}

It is easy to see that $\#: \mathcal P \rightarrow \mathds{N}$ is a bijection. Define the structure $\left\langle \mathcal P, +_{\#}, \times_{\#} \right\rangle$ where the ternary relations $+_{\#}$ and $\times_{\#}$ on $\mathcal P$ are defined as
\begin{align}
(\rho,\sigma,\pi) \in +_{\#}  \quad  &\Leftrightarrow \quad \#\rho + \#\sigma = \#\pi \\
(\rho,\sigma,\pi) \in \times_{\#} \quad  &\Leftrightarrow \quad \#\rho \cdot \#\sigma = \#\pi
\end{align}
The goal is to establish the equality of definable relations $\mathrm{Def}(\mathcal P,\leq,[1]+[1]) = \mathrm{Def}(\mathcal P, +_{\#},\times_{\#})$. For this end, we shall make use of the expressive power of definable relations in arithmetic. 

A relation $R \subseteq \mathds{N}^{k}$ is \emph{recursive} if there exists a Turing machine which always halts and accepts exactly the elements in the relation $R$. A fundamental result \cite[Chapter 10.6 ]{flum} about the definable relations in arithmetic states that for any recursive relation $R \subseteq \mathds{N}^{k}$, there is a first-order formula $\phi_{R}(x_1,\ldots,x_k)$ over the structure $\left\langle \mathds{N}, +, \times \right\rangle$ such that 
\[
(n_1,\ldots,n_k) \in R \quad \Leftrightarrow \quad \left\langle \mathds{N}, +, \times \right\rangle \vDash \phi_{R}(n_1,\ldots,n_k).
\] 
This gives a potent flexibility in determining definable relations; for example, recursive sets would include the above enumeration of primes, any particular fixed non-negative integer $\{n\}$, and a ternary relation $\texttt{Primexp} \subseteq \mathds{N}^3$ such that $(i,m,n) \in \texttt{Primexp}$ if and only if the $i$-th prime appears in the prime factorization of $n$ with exponent $m$. It is not to difficult to see that there is a Turing machine $T_{\mathrm{ord}}$ which always halts that can take a pair of non-negative integers $(m,n)$, compute their prime factorizations and the corresponding canonical representations for partitions $\sigma$ and $\pi$ such that $\#\sigma=m$, $\#\pi=n$ and then verifies if the conditions in Eqn.~\ref{eqn:1} are satisfied.  If we let $\texttt{ord} \subseteq \mathds{N}^2$ be the recursive relation determined by $T_{\mathrm{ord}}$, then the associated first-order formula $\phi_{\mathrm{ord}}$ defines the partial-order $\leq$ over $\mathcal P$. Altogether, we see that $\mathrm{Def}(\mathcal P,\leq,[1]+[1]) \subseteq \mathrm{Def}(\mathcal P, +_{\#},\times_{\#})$.

For the reverse inclusion, we require the interpretation of arithmetic developed in the previous section. Let \texttt{Add} denote the relation defined in Proposition~\ref{eaddition} by the formula $\phi_{\texttt{Add}}(x,y,z)$ and \texttt{Mult} the relation defined in Proposition~\ref{emultipli} by the formula $\phi_{\texttt{Mult}}(x,y,z)$. The interpretation of the arithmetic operations has the pleasing property that 
\begin{align}\label{Eqn:7}
m+n=r  \quad  &\Leftrightarrow \quad ([m],[n],[r]) \in \texttt{Add} \\ 
m\cdot n=r \quad  &\Leftrightarrow \quad ([m],[n],[r]) \in \texttt{Mult} 
\end{align}
Once addition and multiplication are interpreted, we have by the standard process (\cite{flum}) a translation between the first-order formulas in arithmetic and $\textbf{Y}^{\ast}$ which we formalize in the following lemma:

\begin{lemma}\label{lem:translate}
For any first-order formula $\phi(x_1,\ldots,x_k)$ in the language $(+,\times)$ of arithmetic, there is first-order formula $\psi_{\phi}(x_1,\ldots,x_k)$ in the language $(\leq,[1]+[1])$ of $\textbf{Y}^{\ast}$ such that
\begin{align}
\left\langle \mathds{N}, +, \times \right\rangle \vDash \phi(n_1,\ldots,n_k) \quad \Leftrightarrow \quad \textbf{Y}^{\ast}\vDash \psi_{\phi}([n_1],\ldots,[n_k])
\end{align}
for all $n_1,\ldots,n_k \in \mathds{N}$.
\end{lemma}

We have an interpretation of arithmetic over the total partitions definable in the language of Young's lattice with a constant $[1]+[1]$, and another copy of arithmetic over $\mathcal P$ determined by the pull-back relations $+_{\#}$, $\times_{\#}$; therefore, in order to complete the argument we must show that the arithmetization $\#: \mathcal P \rightarrow \mathds{N}$ itself can be definably encoded in the total partitions.

\begin{proposition}\label{prop:tran}
There is a first-order formula $\Psi_{\mathrm{tran}}(x,y)$ in the language $(\leq, [1]+[1])$ of $\textbf{Y}^{\ast}$ such that for any $\sigma,\pi \in \mathcal P$,
\begin{align*}
\textbf{Y}^{\ast} \vDash \Psi_{\mathrm{tran}}(\sigma,\pi) \quad \Leftrightarrow \quad \pi = [\#\sigma].
\end{align*}
\end{proposition}
\begin{proof}
Let $\phi_{\texttt{Primexp}}$ be the arithmetic formula which defines the ternary relation $\texttt{Primexp} \subseteq \mathds{N}^{3}$ and $\psi_{\texttt{Primexp}}$ be the translation given by Lemma~\ref{lem:translate}. Then we see that $\pi = [\#\sigma]$ if and only if $\pi$ is total, and
\begin{enumerate}

	\item if $\sigma=\emptyset$, then $\pi=\emptyset$, and

	\item if $\sigma=m[1]$ for some $m$, then for all total partitions $[i],[k]$, 
	\[
	\psi_{\texttt{Primexp}}([i],[k],\pi) \rightarrow [i]=[1] \mathrel{\wedge} [k]=[m-1],
	\] 
	and 
	
	\item if $\sigma \neq m[1]$ or $\sigma \neq \emptyset$, then for all total partitions $[i],[m]$, $\psi_{\texttt{Primexp}}([i],[m],\pi)$ if and only if $[i]$ is a block of $\sigma$ which appears with frequency $m$. 
	
\end{enumerate}
Correctness follows from Lemma~\ref{lem:translate} and the definition of $\#: \mathcal P \rightarrow \mathds{N}$.
\end{proof}

The ternary relations $+_{\#}$ and $\times_{\#}$ can then be recovered by the formulas
\begin{align}
\exists x^{\ast} \exists y^{\ast} \exists z^{\ast} & \Psi_{\mathrm{tran}}(x,x^{\ast}) \wedge \Psi_{\mathrm{tran}}(y,y^{\ast}) \wedge \Psi_{\mathrm{tran}}(z,z^{\ast}) \wedge \phi_{\texttt{Add}}(x^{\ast},y^{\ast},z^{\ast}) \\
\exists x^{\ast} \exists y^{\ast} \exists z^{\ast} & \Psi_{\mathrm{tran}}(x,x^{\ast}) \wedge \Psi_{\mathrm{tran}}(y,y^{\ast}) \wedge \Psi_{\mathrm{tran}}(z,z^{\ast}) \wedge \phi_{\texttt{Mult}}(x^{\ast},y^{\ast},z^{\ast})
\end{align}
which can be seen using Proposition~\ref{prop:tran} and the properties displayed in Eqn.(4.2) and Eqn.(4.3). This completes the demonstration of the following theorem:

\begin{theorem}\label{thm:mdp}
$\textbf{Y}^{\ast} = \left\langle \mathcal P, \leq, [1]+[1] \right\rangle$ has the maximal definability property.
\end{theorem}


\section{Complexity of some first-order fragments?}\label{sec:5}

Theorem~\ref{thm:sigma1} establishes that the $\Sigma_4$-theory of Young's lattice with the single constant $[1]+[1]$ added to the language is undecidable, but does so by an interpretation of arithmetic which may be too expensive in the complexity of formulas involved. It may be that undecidability persists in less formally complex fragments of the ordering.

The satisfaction of a $\Sigma_1$-sentence in $\textbf{Y} = \left\langle \mathcal P, \leq \right\rangle$ asserts the existence of a certain poset embedded in $\textbf{Y}$. Since $\textbf{Y}$ contains a non-trivial cover, it follows that the $\Sigma_1$-theory is at least \textbf{NP}-hard if it is decidable.

\begin{conjecture}\label{conj:1}	
The $\Sigma_1$-theory of Young's lattice $\textbf{Y} = \left\langle P, \leq \right\rangle$ is \textbf{NP}-complete.
\end{conjecture}

Stated in an alternate manner, is there an \textbf{NP}-characterization of the posets which embed in Young's lattice? If we allow constants for all partitions, how much does the resulting $\Sigma_n$-theory differ from the theory without constants?  Let $\left\langle \mathcal P, \leq, \pi : \pi \in \mathcal P \right\rangle$ denote Young's lattice where we have added every partition has a constant to the language. 

Using the arithmetic interpretation in Section~\ref{sec:3} and the definability of recursive relations referenced in Section~\ref{sec:4}, it follows that for every partition $\pi \in \mathcal P$ there is a $\Sigma_4$-formula in $\textbf{Y}^\ast$ which is uniquely satisfied by $\pi$. Then every formula in $\left\langle \mathcal P, \leq, \pi : \pi \in \mathcal P \right\rangle$ utilizing constants $\pi_1,\ldots,\pi_n$ is logically equivalent to a formula in $\textbf{Y}^{\ast}$ at the complexity expense of a finite disjunction of $\Sigma_4$-formulas representing the constants $\pi_1,\ldots,\pi_n$.

\begin{corollary}
For $n \geq 4$, the $\Sigma_n$-definable relations of $\textbf{Y}^{\ast}$ and $\left\langle \mathcal P, \leq, \pi : \pi \in \mathcal P \right\rangle$ are the same.
\end{corollary}

It may may be that with all constants, undecidability of the $\Sigma_n$-theory arises at the earliest possible instance. If $\exists \bar{x} \phi(\bar{x},\pi_1,\ldots,\pi_k)$ is a $\Sigma_1$-sentence utilizing the constants $\{\pi_1,\ldots,\pi_k\}$, then the restriction of $\textbf{Y}$ to the constants determines a poset $P(\pi_1,\ldots,\pi_k)$ and the satisfaction of $\exists \bar{x} \phi(\bar{x},\pi_1,\ldots,\pi_k)$ asserts the existence of a subposet in $\textbf{Y}$ which extends $P(\pi_1,\ldots,\pi_k)$.

\begin{conjecture}\label{conj:2}						
The $\Sigma_1$-theory of Young's lattice with constants $\left\langle \mathcal P, \leq, \pi : \pi \in \mathcal P \right\rangle$ is undecidable.			
\end{conjecture}

The conjectures are motivated by the fact that they have affirmative answers for the subword order on finite alphabets (\cite[Thm 3.3]{halfon} and \cite[Prop 2.2]{kuske}). The ability to use constants and the subword order in building first-order formulas allows for a $\Sigma_1$-interpretation of natural arithmetic which establishes the analogue of Conjecture~\ref{conj:2} in \cite[Thm 3.3]{halfon}. We suspect this approach can not succeed in Young's lattice, but Conjecture~\ref{conj:2} may still be established by an interpretation with a weaker theory. 

It is unclear the expressive power one gains after adjoining the constant [1]+[1] to the language, and in light of Theorem~\ref{thm:sigma1} we leave the following questions:

\begin{question}
What is the complexity of the $\Sigma_n$-theory of $\textbf{Y}^{\ast}$, for $n=1,2,3$?
\end{question}

\begin{question}
What is the complexity of the $\Sigma_n$-theories of $\textbf{Y}$, for $n=2,3$?
\end{question}

In \cite{halfon}, it shown that in the pure subword order over a two-letter alphabet the $\Sigma_2$-theory is undecidable, but we hesitate to conjecture that the same remains true in Young's lattice.

Finally, we end with the question of how much of the previous development can be carried out for general differential posets. R.P. Stanley ends his paper \cite[Problem 1]{stanley} intuiting doubt that a ``reasonable'' characterization of differential posets is possible, and perhaps so, but all methods known to this author for generating differential posets suggests the following may still be possible to establish:

\begin{conjecture}
If $\textbf{P}$ is a nontrivial differential poset, then the elementary theory is undecidable and non-finitely axiomatizable.
\end{conjecture}


\begin{acknowledgements}							
I would like to thank R.S. Thinniyam for introducing me to the papers \cite{halfon,kuske,kudinov} and work on the complexity of the various word orders.				
\end{acknowledgements}



\begin{thebibliography}{99}



\bibitem{univ}
  S. Burris, H.P. Sankappanavar:
  A Course in Universal Algebra.
  Graduate Texts in Mathematics, vol.78, Springer-Verlag, New York, 1981.


\bibitem{flum}
  H.-D. Ebbinghaus, J. Flum, W. Thomas:
  Mathematical Logic: 2nd edition.
  UTM, Springer-Verlag, New York, 1994.



\bibitem{halfon}
  S. Halfon, P. Schnoebelen, G. Zetzsche:
  \emph{Decidability, complexity, and the epressiveness of first-order logic over the subword ordering}.
  arxiv arXiv:1701.07470v1 [cs.LO]
  (2017)

\bibitem{kuske}
  D. Kuske:
  \emph{Theories of orders on the set of words}.
  RAIRO-inf. Theor. Appl.,
  \textbf{40}:53-74,
  (2006)


\bibitem{semilat}
  J. Je\v{z}ek, R. Mckenzie:
  \emph{Definability in Substructure Orderings I: finite semilattices}.
  Algebra Universalis,
  \textbf{61}:59-75,
  (2009)

\bibitem{poset}
  J. Je\v{z}ek, R. Mckenzie:
  \emph{Definability in Substructure Orderings II: finite ordered sets}.
  Order,
  \textbf{27}:115-145,
  (2010)

\bibitem{distr}
  J. Je\v{z}ek, R. Mckenzie:
  \emph{Definability in Substructure Orderings III: finite distributive lattices}.
  Algebra Universalis,
  \textbf{61}:283-300,
  (2009)

\bibitem{latt}
  J. Je\v{z}ek, R. Mckenzie:
  \emph{Definability in Substructure Orderings IV: finite lattices}.
  Algebra Universalis,
  \textbf{61}:301-312,
  (2009)


\bibitem{kudinov}
  O.V. Kudinov, V.L. Selivanov, L.V. Yartseva:
  \emph{Definability in the Subword Order}.
  In: Proc. CiE 2010, LNCS 6158. Springer, 
  pp.246-255
	(2010)
	

\bibitem{malcev}
  A.I. Mal'cev:
  \emph{Universally axiomatizable subclasses of locally finite classes of models}.
  Sibirsk. Mat.,
  \textbf{8}(3):1005-1014,
  (1967)



\bibitem{matiy}
  Y. Matiyasevich:
  \emph{Hilbert's Tenth Problem}.
  MIT Press
  (1993)


\bibitem{eqrec}
  O. Pretzel, J. Siemons:
  \emph{Reconstruction of partitions(English Summary)}.
  Electron. J. Combin.,
  \textbf{11}(2), Note 5
  (2004/06)



\bibitem{stanley}
  R.P. Stanley:
  \emph{Differential posets}.
  J. of the Amer. Math. Soc.,
  \textbf{1}(4), 919-961
  (1988)


\bibitem{stock}
  P.K. Stockmeyer:
  \emph{The falsity of the reconstruction conjecture for tournaments}.
  J. Graph Theory,
  \textbf{1}, 19-25
  (1977)



\bibitem{suter}
  R. Suter:
  \emph{Young's lattice and dihedral symmetries}.
  Electron. J. Combin.,
  \textbf{23}(2), 233-238
  (2002)



\bibitem{interp}
  A. Tarski, A. Mostowski, R. Robinson:
  \emph{Undecidable Theories},
  North-Holland,
  Amsterdam, 1953.
	

	
	\bibitem{thinn3}
  R.S. Thinniyam:
  \emph{Defining Recursive Predicates in Graph Orders}.
  Logical Methods in Computer Science
	\textbf{14}(3:21), 1-38
  (2018)
	
	
	\bibitem{wires}
  A. Wires:
  \emph{Definability in Substructure Ordering of Simple Graphs}.
  Annals of Combinatorics,
  \textbf{20}(1):139-176,
  (2016)
	





\end{thebibliography}
\end{document}